\documentclass{amsart}

\usepackage{amssymb,latexsym,amsfonts,amsmath}
\usepackage{graphicx,psfrag,epsfig,epsf}
\include{diagrams}
\usepackage{eso-pic}
\usepackage{graphicx}
\usepackage{color}
\usepackage{diagrams}
\usepackage{algorithm2e}
\usepackage{tikz}
\usetikzlibrary{automata}
\usepackage{subfigure}
\usepackage{amsfonts}
\usepackage{amssymb,latexsym,amsfonts,amsmath}
\usepackage{graphicx}

\newtheorem{theorem}{Theorem}[section]
\newtheorem{lemma}[theorem]{Lemma}

\newtheorem{proposition}[theorem]{Proposition}

\theoremstyle{definition}
\newtheorem{definition}[theorem]{Definition}

\theoremstyle{remark}
\newtheorem{remark}[theorem]{Remark}
\numberwithin{equation}{section}

\begin{document}

\begin{abstract}                          % Abstract of not more than 200 words.
In the last few years there has been a growing interest in the use of symbolic models for the formal verification and control design of purely continuous or hybrid systems. Symbolic models are abstract descriptions of continuous systems where one symbol corresponds to an "aggregate" of continuous states. In this paper we face the problem of deriving symbolic models for nonlinear control systems affected by disturbances. The main contribution of this paper is in proposing symbolic models that can be effectively constructed and that approximate nonlinear control systems affected by disturbances in the sense of alternating approximate bisimulation.
\end{abstract}

\title[Symbolic models for nonlinear control systems affected by disturbances]{Symbolic models for nonlinear control systems\\affected by disturbances}
\thanks{This work has been partially supported by the Center of Excellence for Research DEWS, University of L'Aquila, Italy.}

\author[Alessandro Borri, Giordano Pola and Maria Domenica Di Benedetto]{
Alessandro Borri$^{1}$, Giordano Pola$^{1}$ and Maria Domenica Di Benedetto$^{1}$}
\address{$^{1}$
Department of Electrical and Information Engineering, Center of Excellence DEWS,
University of L{'}Aquila, Poggio di Roio, 67040 L{'}Aquila, Italy}
\email{ \{alessandro.borri,giordano.pola,mariadomenica.dibenedetto\}@univaq.it}

\maketitle

\section{Introduction}
An emerging trend in the control systems and computer science communities is the use of symbolic models for the analysis and control design of purely continuous or hybrid systems \cite{IEEETACCCS}. Symbolic models are abstract descriptions of continuous systems where each symbol corresponds to an "aggregate" of continuous states \cite{paulo}. The use of symbolic models provides a formal approach to solve control problems in which software and hardware interact with the physical world. Moreover, it provides the designer with a systematic method to address a wide spectrum of novel specifications that are difficult to enforce by means of conventional control design paradigms. Examples of such specifications include logic specifications expressed in linear temporal logic or automata on infinite strings. \\
%During the last years, several classes of dynamical and control systems admitting symbolic models were identified. We recall from \cite{DiscAbs} timed, multi--rate, rectangular automata, and o-minimal hybrid systems in the class of hybrid automata. Control systems were addressed in \cite{LTLControl}, \cite{HCS06} and \cite{Belta:06}, where symbolic models were shown to exist for controllable discrete--time linear systems, piecewise--affine systems and multi--affine systems, respectively. Most of the aforementioned work is based on the notions of simulation and bisimulation, as introduced by Milner \cite{Milner} and Park \cite{Park}. Insights into the construction of symbolic models for continuous and hybrid systems have been recently gained by the notion of approximate bisimulation \cite{AB-TAC07}. Based on this notion, incrementally stable nonlinear control systems were shown in \cite{PolaAutom2008} to admit symbolic models. This result has been further generalized to nonlinear switched systems in \cite{GirardTAC2010} and nonlinear time--delay systems in \cite{PolaSCL10,PolaCDC10}. In the aforementioned work, control systems are supposed to be not affected by exogenous disturbance inputs. However, in many realistic situations, physical processes are characterized by a certain degree of uncertainty that is often modeled by disturbance inputs. \\
The literature on symbolic models is very broad and includes results on timed automata \cite{alur}, rectangular hybrid automata \cite{puri} and o-minimal hybrid systems \cite{lafferriere, brihaye}. Early results for classes of control systems were based on dynamical consistency properties \cite{caines}, natural invariants of the control system \cite{koutsoukos}, $l$-complete approximations \cite{moor} and quantized inputs and states \cite{forstner,BMP02}. Recent results include work on piecewise-affine and multi-affine systems \cite{habets,BH06}, set-oriented discretization approach for discrete-time nonlinear optimal control problem \cite{junge1} and abstractions based on convexity of reachable sets for sufficiently small sampling time \cite{gunther}. Symbolic models for nonlinear control systems, time--delay systems and switched systems based on the notions of approximate bisimulation \cite{AB-TAC07} and incremental stability \cite{IncrementalS} have been studied in \cite{PolaAutom2008,PolaSIAM2009}, \cite{PolaSCL10,PolaCDC10} and \cite{GirardTAC2010}.\\ 
In this paper we face the problem of deriving symbolic models for nonlinear control systems affected by disturbances. The presence of disturbances requires us to replace the notion of approximate bisimulation employed in \cite{PolaAutom2008,GirardTAC2010,PolaSCL10} with the notion of alternating approximate bisimulation introduced in \cite{PolaSIAM2009} and  inspired by Alur and coworkers' alternating bisimulation \cite{Alternating}. As discussed in \cite{PolaSIAM2009,paulo} this notion is a key ingredient when constructing symbolic models of systems affected by disturbances because it guarantees that control strategies synthesized on the symbolic models can be readily transferred to the original model. The existence of alternating approximately bisimilar symbolic models for incrementally stable nonlinear control systems affected by disturbances has been proven in \cite{PolaSIAM2009}. However, the results of \cite{PolaSIAM2009} cannot be easily used for the construction of symbolic models because they rely on the computation of sets of reachable states which is a difficult task in general. 
In this work we propose alternative symbolic models to the ones proposed in \cite{PolaSIAM2009} which are proven to be effectively computable. The key ingredient in our results is the derivation of finite approximations of the disturbance input functional space by resorting to spline analysis \cite{splineBook}. Spline analysis has been also employed in \cite{PolaSCL10,PolaCDC10} for deriving symbolic models of time--delay systems. As discussed in the paper, the approximation scheme proposed in \cite{PolaSCL10,PolaCDC10} cannot be used in this framework because it would lead to symbolic models that cannot be effectively constructed. For this reason in this paper we elaborate alternative spline--based approximation schemes for the disturbance input functional space which instead guarantee the effective computation of the proposed symbolic models. 
The main contribution of this paper lies in showing that: 

\bigskip
\textit{If the control system is incrementally stable and the disturbance input signals are bounded and Lipschitz continuous then symbolic models can be effectively constructed which are shown to be alternating approximately bisimilar to the original control systems with any desired accuracy}.
\bigskip

A preliminary version of this work appeared in the conference publication \cite{BorriCDC2011}. This paper is organized as follows. Preliminary definitions are recalled in Section \ref{sec2}. In Section \ref{sec3} we propose a spline--based approximation scheme for the disturbance input functional space. In Section \ref{sec4} we show how to construct symbolic models that approximate nonlinear control systems affected by disturbances in the sense of alternating approximate bisimulation. Section \ref{sec5} shows an illustrative example. Finally Section \ref{sec6} offers some concluding remarks. 

\section{Preliminary definitions}\label{sec2}

\subsection{Notation}
A singleton is a set containing exactly one element. The identity map on a set $A$ is denoted by $1_{A}$. Given two sets $A$ and $B$, if $A$ is a subset of $B$ we denote by $1_{A}:A\hookrightarrow B$ or simply by $\imath$ the natural inclusion map taking any $a\in A$ to $\imath (a)  =a\in B$. Given a function $f:A\rightarrow B$ the symbol $f(A)$ denotes the image of $A$ through $f$, i.e. $f(A):=\{b\in B:\exists a\in A$ s.t. $b=f(a)\}$; if $C\subset A$ we denote by $f|_{C}$ the restriction of $f$ to $C$, i.e. $f|_{C}(x):=f(x)$ for any $x\in C$. Given a relation $\mathcal{R}\subseteq A\times B$, the symbol $\mathcal{R}^{-1}$ denotes the inverse relation of $\mathcal{R}$, i.e. $\mathcal{R}^{-1}:=\{(b,a)\in B\times A:( a,b)\in \mathcal{R}\}$; we set $\mathcal{R}(A)=\{b\in B | \exists a\in A \text{ s.t. } (a,b)\in \mathcal{R}\}$ and $\mathcal{R}^{-1}(B)=\{a\in A | \exists b\in B \text{ s.t. } (a,b)\in \mathcal{R}\}$. The symbols $\mathbb{N}$, $\mathbb{Z}$, $\mathbb{R}$, $\mathbb{R}^{+}$ and $\mathbb{R}_{0}^{+}$ denote the set of natural, integer, real, positive real, and nonnegative real numbers, respectively. Given a vector $x\in\mathbb{R}^{n}$, we denote by $\Vert x\Vert$ the infinity norm of $x$.
Given a measurable function \mbox{$f:\mathbb{R}_{0}^{+}\rightarrow\mathbb{R}^n$}, the (essential) supremum of $f$ is denoted by $\Vert f\Vert_{\infty}$. 
Given $\mu\in\mathbb{R}^{+}$ and $A\subseteq \mathbb{R}^{n}$, we denote by $\mu A$ the set $\{b\in \mathbb{R}^{n}\,|\,\exists a\in A \text{ s.t. } b=\mu a \}$. A continuous function \mbox{$\gamma:\mathbb{R}_{0}^{+}\rightarrow\mathbb{R}_{0}^{+}$} is said to belong to class $\mathcal{K}$ if it is strictly increasing and
\mbox{$\gamma(0)=0$}; function $\gamma$ is said to belong to class $\mathcal{K}_{\infty}$ if \mbox{$\gamma\in\mathcal{K}$} and $\gamma(r)\rightarrow\infty$
as $r\rightarrow\infty$. A continuous function \mbox{$\beta:\mathbb{R}_{0}^{+}\times\mathbb{R}_{0}^{+}\rightarrow\mathbb{R}_{0}^{+}$} is said to belong to class $\mathcal{KL}$ if, for each fixed $s$, the map $\beta(r,s)$ belongs to class $\mathcal{K}_{\infty}$ with respect to $r$ and, for each fixed $r$, the map $\beta(r,s)$ is decreasing with respect to $s$ and $\beta(r,s)\rightarrow0$ as \mbox{$s\rightarrow\infty$}. The symbol $C^{0}([0,\tau];Y)$ denotes the set of continuous functions from a closed interval of the form $[0,\tau]$ with $\tau\in\mathbb{R}^{+}$ to a set $Y\subseteq \mathbb{R}^{m}$. 
Consider a bounded set $A \subseteq \mathbb{R}^n$ with interior. Let $H=[a_1,b_1]\times[a_2,b_2]\times \dots \times [a_n,b_n]$ be the smallest hyperrectangle containing $A$ and set $\hat{\mu}_{A}=\min_{i=1,2,\dots,n} (b_i-a_i)$. It is readily seen that for any $\mu \leq \hat{\mu}_A$ and any $a\in A$ there always exists $b\in (2\mu\mathbb{Z}^{n})\cap A$ such that $\Vert a-b \Vert \leq \mu$.

\subsection{Control systems and incremental stability}\label{II.B}

In this paper we consider the following nonlinear control system:
\begin{equation}
\label{control_system}
\dot{x}=f(x,u,d),
\end{equation}
where $x\in X\subseteq \mathbb{R}^{n}$ is the state, $u\in U\subseteq \mathbb{R}^{m}$ and $d\in D\subseteq \mathbb{R}^{l}$ are the control and disturbance inputs. We suppose that $f(0,0,0)=0$, the set $X$ is convex with the origin as an interior point and the sets $U$ and $D$ are compact, convex, with the origin as an interior point. Control input functions are supposed to belong to the set $\mathcal{U}$ of piecewise--constant functions of time from intervals of the form \mbox{$]a,b[\subseteq\mathbb{R}$} to $U$. Disturbance input functions are supposed to belong to the set $\mathcal{D}$ of continuous functions of time of the form \mbox{$d:]a,b[\subseteq\mathbb{R}\rightarrow D$} satisfying the following Lipschitz assumption: there exists $\kappa_{d} \in\mathbb{R}^{+}$ such that:
\begin{equation}
\Vert d(t_2)-d(t_1)\Vert \leq\kappa_{d} \vert t_2-t_1 \vert,
\label{eq:Lipschitz}
\end{equation}
for any $d\in\mathcal{D}$ and $t_1,t_2\in ]a,b[$. Function $f:\mathbb{R}^{n}\times U \times D \rightarrow\mathbb{R}^{n}$ is continuous and enjoys the following Lipschitz assumption: for every compact set \mbox{$K\subset\mathbb{R}^{n}$}, there exists a constant \mbox{$k\in\mathbb{R}^{+}$} such that 
\[
\Vert f(x,u,d)-f(y,u,d)\Vert \leq k \Vert x-y\Vert,
\]
for all $x,y\in K$, $u\in U$ and $d\in D$. In the sequel, we refer to the nonlinear control system in (\ref{control_system}) by means of the tuple:
\begin{equation}
\Sigma=(X,\mathcal{U},\mathcal{D},f),
\label{tuple}
\end{equation}
where each entity has been defined above. Since control inputs are piecewise--constant, system $\Sigma$ is often referred to in the literature as a  \textit{nonlinear sample--data control system}, see e.g. \cite{Nesic01}. \\
A curve \mbox{$\xi:]a,b[\rightarrow\mathbb{R}^{n}$} is said to be a \textit{trajectory} of $\Sigma$ if there exist $u\in\mathcal{U}$ and $d\in\mathcal{D}$ satisfying 
\[
\dot{\xi}(t)=f(\xi(t),u(t),d(t)),
\]
for almost all $t\in$ $]a,b[$. Although we have defined trajectories over open domains, we shall refer to trajectories \mbox{${\xi:}[0,\tau]\rightarrow\mathbb{R}^{n}$} defined on closed domains $[0,\tau],$ $\tau\in\mathbb{R}^{+}$ with the understanding of the existence of a trajectory \mbox{${\xi}^{\prime}:]a,b[\rightarrow\mathbb{R}^{n}$} such that \mbox{${\xi}={\xi}^{\prime}|_{[0,\tau]}$}. We also write $\xi_{xud}(t)$ to denote the point reached at time $t$ under the control input $u$ and disturbance input $d$ from initial condition $x$; this point is uniquely determined, since the assumptions on $f$ ensure existence and uniqueness of trajectories \cite{Sontag}. A control system $\Sigma$ is said to be forward complete if every trajectory is defined on an interval of the form $]a,\infty\lbrack$. Sufficient and necessary conditions for a system to be forward complete can be found in \cite{fc-theorem}. 
In the sequel, we will make use of the following stability notion.

\begin{definition}
\label{dISS}
\cite{IncrementalS} 
A control system $\Sigma$ is incrementally input--to--state stable ($\delta$--ISS) if it is forward complete and there exist a $\mathcal{KL}$ function $\beta$ and two $\mathcal{K}_{\infty}$ functions $\gamma_{u}$ and $\gamma_{d}$ such that for any $t\in{\mathbb{R}_{0}^{+}}$, any $x_{1},x_{2}\in{\mathbb{R}^{n}}$, any
$u_{1},u_{2}\in\mathcal{U}$ and any $d_{1},d_{2}\in\mathcal{D}$, the following inequality is satisfied:
\begin{equation*}
\Vert \xi_{x_{1}u_{1}d_{1}}(t)-\xi_{x_{2}u_{2}d_{2}}(t)\Vert\leq 
\beta(\Vert x_{1}-x_{2}\Vert ,t)
+\gamma_{u}(\Vert u_{1}-u_{2}\Vert_{\infty})
+\gamma_{d}(\Vert d_{1}-d_{2}\Vert_{\infty}).
\end{equation*}
\end{definition}

The above incremental stability notion can be characterized in terms of dissipation inequalities, as follows.
\begin{definition}
\label{dISS_Lyapunov}
\cite{IncrementalS} 
A smooth function $V:\mathbb{R}^{n}\times\mathbb{R}^{n}\rightarrow\mathbb{R}$ is called a $\delta$--ISS Lyapunov function for a control system $\Sigma=(X,\mathcal{U},\mathcal{D},f)$ if there exist $\lambda\in\mathbb{R}^{+}$ and $\mathcal{K}_{\infty}$ functions $\underline{\alpha}$, $\overline{\alpha}$, $\sigma_{u}$ and $\sigma_{d}$ such that for any $x_{1},x_{2}\in X$, any $u_{1},u_{2}\in U$ and any $d_{1},d_{2}\in D$ the following conditions hold true:
\begin{itemize}
\item[(i)] $\underline{\alpha}(\Vert{x_{1}-x_{2}}\Vert)\leq V(x_{1},x_{2})\leq\overline{\alpha}(\Vert{x_{1}-x_{2}}\Vert)$,
\item[(ii)] $\frac{\partial{V}}{\partial{x_{1}}} f(x_{1},u_{1},d_{1})+\frac{\partial{V}}{\partial{x_{2}}} f(x_{2},u_{2},d_{2}) \leq 
-\lambda V(x_{1},x_{2}) + \sigma_{u}(\Vert{u_{1}-u_{2}}\Vert)+\sigma_{d}(\Vert{d_{1}-d_{2}}\Vert)$.
\end{itemize}
\end{definition}
The following result adapted from \cite{IncrementalS} completely characterizes $\delta$--ISS in terms of existence of $\delta$--ISS Lyapunov functions.
\begin{theorem}
\label{TH-IISS}
The control system $\Sigma$ in (\ref{tuple}) is $\delta$--ISS if and only if it admits a $\delta$--ISS Lyapunov function.
\end{theorem}

%In the sequel we assume that control systems are $\delta$--ISS. Backstepping techniques for the incremental stabilization of nonlinear control systems have been recently established in \cite{ZamaniTAC2011}.

\subsection{Transition systems and approximate equivalence notions}
We will use alternating transition systems \cite{Alternating} to describe both control systems as well as their symbolic models. 

\begin{definition}
An (alternating) transition system $T$ is a quintuple: 
\[
T=(Q,L,\rTo,O,H),
\]
consisting of:
\begin{itemize}
\item a set of states $Q$;
\item a set of labels $L=A\times B$, where:
\begin{itemize}
\item[--] $A$ is the set of control labels,
\item[--] $B$ is the set of disturbance labels;
\end{itemize}
\item a transition relation $\rTo \subseteq Q\times L\times Q$;
\item a set of outputs $O$;
\item an output function $H:Q\rightarrow O$. 
\end{itemize}
A transition $(q,(a,b),q^{\prime})\in\rTo$ is denoted by $q\rTo^{(a,b)}q^{\prime}$. 
A state run of $T$ is a sequence of transitions:
\begin{equation}
q_{1} \rTo^{(a_{1},b_{1})} q_{2} \rTo^{(a_{2},b_{2})} \,\, {...}\, \rTo^{(a_{N-1},b_{N-1})} q_{N} .
\label{eq:state_run}
\end{equation}
An output run is a sequence $\{o_i\}_{i\in\mathbb{N}}$ of outputs such that there exists a state run of the form (\ref{eq:state_run}) with $o_i=H(q_i)$, $i=1,2,..., N$. Transition system $T$ is said to be:
\begin{itemize}
\item \textit{countable}, if $Q$ and $L$ are countable sets;
\item \textit{symbolic}, if $Q$ and $L$ are finite sets;
\item \textit{metric}, if the output set $O$ is equipped with a metric $\mathbf{d}:O\times O\rightarrow\mathbb{R}_{0}^{+}$. 
\end{itemize}
\end{definition}

In the sequel we consider bisimulation relations \cite{Milner,Park} to relate properties of control systems and symbolic models. 
Intuitively, a bisimulation relation between a pair of transition systems $T_{1}$ and $T_{2}$ is a relation between the corresponding state sets explaining how a state run $r_{1}$ of $T_{1}$ can be transformed into a state run $r_{2}$ of $T_{2}$, and vice versa. While typical bisimulation relations require that $r_{1}$ and $r_{2}$ have the same output run, i.e. $H_{1}(r_{1}) = H_{2}(r_{2})$, the notion of approximate bisimulation relation, introduced in \cite{AB-TAC07}, relaxes this condition and require that $H_{1}(r_{1})$ is simply close to $H_{2}(r_{2})$, where closeness is measured with respect to a metric on the set of outputs. In this work we consider a generalization of approximate bisimulation, called alternating approximate bisimulation, that has been introduced in \cite{PolaSIAM2009} to relate properties of control systems affected by disturbances and their symbolic models. 

\begin{definition}
\label{ASR_S} 
Consider a pair of metric transition systems $T_1=(Q_{1},A_{1}\times B_{1},\rTo_{1},O_1,H_{1})$ and $T_2=(Q_{2},A_{2}\times B_{2},\rTo_{_2},O_2,H_{2})$ with the same set of outputs $O_1=O_2$ and metric $\mathbf{d}$ and consider a precision $\varepsilon\in\mathbb{R}_{0}^{+}$. A relation 
\[
\mathcal{R}\subseteq Q_{1}\times Q_{2}
\]
is said to be an alternating $\varepsilon$--approximate ($A\varepsilon A$) bisimulation relation between $T_{1}$ and $T_{2}$ if for all $(q_{1},q_{2})\in\mathcal{R}$ the following conditions are satisfied:

\begin{itemize}
\item[(i)]  $\mathbf{d}(H_{1}(q_{1}),H_{2}(q_{2}))\leq\varepsilon$;
\item[(ii)] $\forall a_{1}\in A_{1}$ $\exists a_{2}\in A_{2}$ $\forall b_{2}\in B_{2}$ $\exists
b_{1}\in B_{1}$ such that $q_{1}\rTo_{1}^{(a_{1},b_{1})} q_{1}^{\prime}$, $q_{2}\rTo_{2}^{(a_{2},b_{2})}q_{2}^{\prime}$ and $(q_{1}^{\prime},q_{2}^{\prime} )\in\mathcal{R}$;
\item[(iii)] $\forall a_{2}\in A_{2}$ $\exists a_{1}\in A_{1}$ $\forall b_{1}\in B_{1}$ $\exists
b_{2}\in B_{2}$ such that $q_{1}\rTo_{1}^{(a_{1},b_{1})} q_{1}^{\prime}$, $q_{2}\rTo_{2}^{(a_{2},b_{2})}q_{2}^{\prime}$ and $(q_{1}^{\prime},q_{2}^{\prime})\in\mathcal{R}$.
\end{itemize}

Transition systems $T_{1}$ and $T_{2}$ are alternating $\varepsilon$--approximately ($A\varepsilon A$) bisimilar if there exists an $A\varepsilon A$ bisimulation relation such that $\mathcal{R}(Q_{1})=Q_{2}$ and $\mathcal{R}^{-1}(Q_{2})=Q_{1}$. 
\end{definition}

As discussed in \cite{PolaSIAM2009}, the notion of alternating approximate bisimulation guarantees that control strategies synthesized on symbolic models, based on alternating approximate bisimulations, can be readily transferred to the original model, independently of the particular realization of the disturbance inputs.  When sets $B_{1}$ and $B_{2}$ are singletons, the above notion boils down to approximate bisimulation \cite{AB-TAC07}. When $\varepsilon=0$, the above notion can be viewed as the two-player version of alternating bisimulation \cite{Alternating}.

\section{Spline approximation of the disturbance space}\label{sec3}

One of the key ingredients in the results presented in this paper is the approximation of the disturbance input functional space through spline analysis \cite{splineBook}. In this section we describe this approximation scheme. 
Given a time parameter $\tau\in\mathbb{R}^{+}$, define 
\[
\mathcal{D}_{\tau}:=\{d\in\mathcal{D} | \text{ the domain of } d \text{ is } [0,\tau]\},
\]
and set
\begin{equation}
\label{emme}
M=\sup_{d\in\mathcal{D}_{\tau}}\Vert d \Vert_{\infty}.  
\end{equation}
In the sequel we propose an approximation of the functional space $\mathcal{D}_{\tau}$ in the sense of the following definition.

\begin{definition}
\label{count_approx_def}
A map 
\[
\mathcal{A}:\mathcal{\mathbb{R}}^{+}\rightarrow2^{C^{0}([0,\tau];D)}
\]
is a finite inner approximation of $\mathcal{D}_{\tau}$ if for any desired precision $\theta\in\mathcal{\mathbb{R}}^{+}$ the following properties hold:
\begin{itemize}
\item[(i)] $\mathcal{A}\left(  \theta\right)  $ is a finite set;
\item[(ii)] $\mathcal{A}\left(  \theta\right)\subseteq\mathcal{D}_{\tau}$;
\item[(iii)] for any $d\in$ $\mathcal{D}_{\tau}$ there exists $z\in\mathcal{A}(\theta)$ such that $\Vert y-z\Vert_{\infty}\leq\theta$.
\end{itemize}
\end{definition}

We start by recalling from \cite{splineBook} the notion of spline. Given $N\in \mathbb{N}$ consider the following functions:
\[
\begin{array}{rl}
s_{0}(t)= &
\left\{
\begin{array}
[c]{l}
1-t/h,\,\, t\in [0,h],\\
0,\,\, \textrm{otherwise,} 
\end{array}
\right.
\vspace{4mm}
\\
s_{i}(t)=
&
\left\{
\begin{array}
[c]{ll}
1-i+t/h, \,\, t\in [(i-1)h,ih],\\
1+i-t/h, \,\, t\in [ih,(i+1)h], \qquad i=1,2,...,N,\\
0, \,\, \textrm{otherwise,} 
\end{array}
\right.
\\
\\
\vspace{4mm}
s_{N+1}(t)=
&
\left\{
\begin{array}
[c]{ll}
1+(t-\tau)/h, \,\, t\in [\tau-h,\tau],\\
0, \,\, \textrm{otherwise,} 
\end{array}
\right.
\end{array}
\]
where $h=\tau/(N+1)$. Functions $s_{i}$ called \textit{splines} are used to approximate $\mathcal{D}_{\tau}$. More precisely, the approximation scheme that we propose is based on the following three steps: 
\begin{itemize}
\item We first scale function $d\in\mathcal{D}_{\tau}$ (Figure \ref{fig1}; first panel) to get the function $d_{1}=\rho\, d$ with:
\begin{equation*}
\rho= 1-\max\left\{\frac{\mu}{M},\frac{2\mu (N+1)}{\kappa_{d}\tau}\right\},
\end{equation*}
where $M$ is as in (\ref{emme}), $\kappa_{d}$ is as in (\ref{eq:Lipschitz}) and $\mu\in\mathbb{R}^{+}$ is a suitable quantization parameter whose role will be clear in the sequel. 
\item We then approximate function $d_{1}\in\mathcal{D}_{\tau}$ (Figure \ref{fig1}; second panel) by means of the piecewise--linear function $d_{2}$ (Figure \ref{fig1}; third panel) obtained by the linear combination of the $N+2$ splines $s_{i}$ centered at times $t=i\,h$ with amplitudes\footnote{This second step allows us to approximate the \textit{infinite}-dimensional space $\mathcal{D}_{\tau}$ by means of the \textit{finite}-dimensional space $D^{N+2}$.} $d_{1}(ih)$.
\item We finally approximate function $d_{2}$ by means of function $d_{3}$ (Figure \ref{fig1}; fourth panel) obtained by the linear combination of the $N+2$ splines $s_{i}$ centered at times $t=i\,h$ with amplitudes $d^{i}_{3}$ chosen in the lattice $[D]_{2\mu}=(2\mu\,\mathbb{Z}^{l})\cap D$ and minimizing the distance from\footnote{This third step allows us to approximate the \textit{finite}-dimensional space $D^{N+2}$ by means of the \textit{finite} set $([D]_{2\mu})^{N+2}$.} $d_{2}(ih)$, i.e.
\[
d^{i}_{3}=\arg \min_{d\in [D]_{2\mu}}\Vert d-d_{2}(ih) \Vert .
\]
\end{itemize}

\begin{figure}
\begin{center}
\includegraphics[scale=0.7]{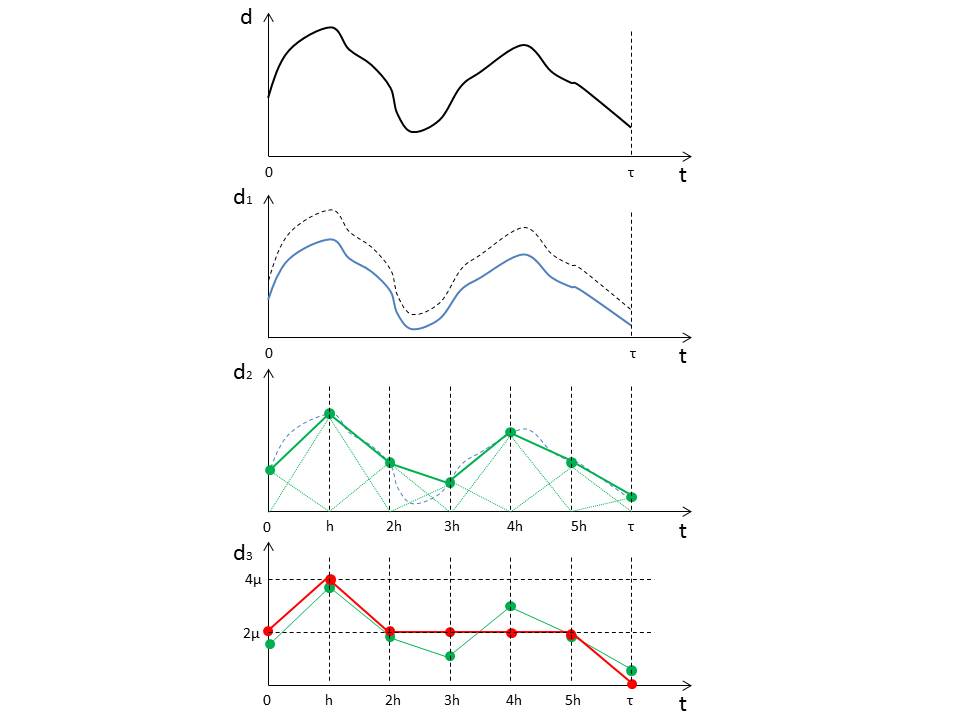}
\caption{Spline--based inner approximation scheme of the disturbance input functional space.} 
\label{fig1}
\end{center}
\end{figure}
Given $N\in\mathbb{N}$ and $\mu\in\mathbb{R}^{+}$ define the following functions:
%\begin{equation}
\begin{align}
%{rcl}
\rho_{\kappa_{d},\tau,M}(N,\mu) & = 1-\max\left\{\frac{\mu}{M},\frac{2\mu }{\kappa_{d} h}\right\},\label{Rho} \\
\Theta_{\kappa_{d},\tau,M}(N,\mu) & = (1- \rho_{\kappa_{d},\tau,M}(N,\mu))M + (1+ \rho_{\kappa_{d},\tau,M}(N,\mu)) \kappa_{d} h +\mu ,\label{Theta}
\end{align}
%\end{equation}
where we recall $h=\tau/(N+1)$. Function $\Theta$ will be shown to be an upper bound of the error associated to the approximation scheme that we propose for $\mathcal{D}_{\tau}$. The following technical result will be useful in the sequel.

\begin{lemma}
\label{inequality}
For any $\theta\in\mathbb{R}^{+}$ there exist $N\in\mathbb{N}$ and $\mu\in\mathbb{R}^{+}$ such that 
\begin{equation}
\label{condQAW}
\begin{array}
{rl}
%$\Theta_{\kappa_{d},\tau,M}(N,\mu)\leq \theta,\,\,\,$ & $\rho_{\kappa_{d},\tau,M}(N,\mu)>0$.
\Theta_{\kappa_{d},\tau,M}(N,\mu)\leq \theta, \qquad \rho_{\kappa_{d},\tau,M}(N,\mu)>0.
\end{array}
\end{equation}
\end{lemma}

\begin{proof}
Choose $\mu=\frac{1}{(N+1)^2}$, $N\in \mathbb{N}$. Function $\rho_{\kappa_{d},\tau,M}(N,\mu)$ in Eq. (\ref{Rho}) rewrites as
\begin{equation*}
\rho_{\kappa_{d},\tau,M}\left(N,\frac{1}{(N+1)^2}\right) = 1-\max\left\{\frac{1}{M (N+1)^2},\frac{2 }{\kappa_{d} \tau (N+1)}\right\}.
\end{equation*}

The right-hand side of the previous equality is increasing with $N$, and it converges to $1$ as $N$ goes to infinity; then it is clear that for a sufficiently large $N$ one gets $\rho_{\kappa_{d},\tau,M}\left(N,\frac{1}{(N+1)^2}\right)>0$.
Furthermore, one can write the following upper-bound for the function $\Theta_{\kappa_{d},\tau,M}(N,\mu)$ in (\ref{Theta}):

\begin{align*}
\Theta_{\kappa_{d},\tau,M}\left(N,\frac{1}{(N+1)^2}\right) & =\max\left\{\frac{1}{(N+1)^2},\frac{2 M}{\kappa_{d} \tau (N+1)}\right\} \\
& + \left(2-\max\left\{\frac{1}{M (N+1)^2},\frac{2 }{\kappa_{d} \tau (N+1)}\right\}\right)\frac{\kappa_{d}\tau}{N+1}+ \frac{1}{(N+1)^2} \\
& \leq \frac{1}{N+1}  \left( \max\left\{\frac{1}{N+1},\frac{2 M}{\kappa_{d} \tau}\right\} + 2 \kappa_{d} \tau + \frac{1}{N+1} \right).
\end{align*} 

The right-hand side of the previous inequality is decreasing with $N$, and goes to zero as $N$ goes to infinity. Hence, the result follows. % also $\lim_{N \rightarrow + \infty} \Theta_{\kappa_{d},\tau,M}\left(N,\frac{1}{(N+1)^2}\right)=0$, which implies that for any $\theta\in\mathbb{R}^{+}$ there exists a sufficiently large $N\in\mathbb{N}$ such that $\Theta_{\kappa_{d},\tau,M}\left(N,\frac{1}{(N+1)^2}\right)\leq \theta$, concluding the proof.
\end{proof}

We are now ready to formally introduce the approximation scheme of the disturbance input functional space.

\begin{definition}
\label{countApprox}
Consider the map 
\[
\mathcal{A}_{\mathcal{D}_{\tau}}:\mathcal{\mathbb{R}}^{+}\rightarrow2^{C^{0}([0,\tau];D)}
\]
that associates to any precision $\theta\in\mathbb{R}^{+}$ the set $\mathcal{A}_{\mathcal{D}_{\tau}}(\theta)$ consisting of the collection of all functions:
\begin{equation}
z(t):=\sum_{i=0}^{N_{\theta}+1} z_{i}s_{i}(t), \hspace{5mm} t\in [0,\tau],
\label{zeta}
\end{equation}
satisfying the following conditions:
\begin{itemize}
\item[(i)] $z_{i}\in (2\mu_{\hat{\theta}}\mathbb{Z}^{l})\cap D$ for any $i=0,1,...,N_{\theta}+1$,
\item[(ii)] $\Vert z_{i+1}- z_{i} \Vert \leq \kappa_{d} \tau /(N_{\theta}+1)$ for any $i=0,1,...,N_{\theta}$,
\end{itemize}
with $\hat{\theta}=\min \{\theta,\hat{\mu}_{D}\}$ where $\hat{\mu}_{D}$ is defined in Section 2.1.
\end{definition}

\begin{remark}
\label{remark2}
Since the set $D$ is compact, the set $(2\mu_{\hat{\theta}}\mathbb{Z}^{l})\cap D$ is finite. Therefore, the set $\mathcal{A}_{\mathcal{D}_{\tau}}(\theta)$ is composed of a finite number of functions that can be effectively computed. 
\end{remark}

The following technical result will be used in the sequel.

\begin{lemma}
\label{inner}
For any $\theta\in\mathbb{R}^{+}$, $\mathcal{A}_{\mathcal{D}_{\tau}}(\theta) \subseteq \mathcal{D}_{\tau}$.
\end{lemma}

\begin{proof}
In order to show that any function $z$ in (\ref{zeta}) is in $\mathcal{D}_{\tau}$, we need to show that $z$ enjoys the Lipschitz condition (\ref{eq:Lipschitz}) and $\Vert z \Vert _{\infty} \leq M$. Since $z$ is continuous and defined over the interval $[0,\tau]$, by the triangle inequality it suffices to show that (\ref{eq:Lipschitz}) holds for any $t_1,t_2\in [ih,(i+1)h]$, $i=0,...,N_{\theta}$.
By Eq. (\ref{zeta}) and the definition of spline, the function $z$ is piecewise-linear and is linear in the interval $[ih,(i+1)h]$, with $z(ih)=z_i$. Hence one can write for any $t_{1},t_{2}\in [ih,(i+1)h]$:
\begin{equation}
\frac{\Vert z(t_2)-z(t_1)\Vert}{\Vert t_2-t_1\Vert }=\frac{\Vert z((i+1)h)-z(ih)\Vert}{h}=\frac{\Vert z_{i+1}-z_i)\Vert}{ h }\leq \kappa_{d},
\end{equation}
where the last step holds by condition (ii) in Definition \ref{countApprox}, concluding the proof of the Lipschitz condition. We next show that the boundedness condition holds as well.
Since $z$ is piecewise-linear, $\Vert z \Vert _{\infty}=\max_{i=0,...,N_{\theta}+1} \Vert z_i \Vert$, hence we just need to show that $\Vert z_i \Vert \leq M$ for all $i$. From condition (i) in Definition \ref{countApprox}, $z_{i}\in D$, implying from (\ref{emme}) that $\Vert z_i \Vert \leq M$, which concludes the proof.
\end{proof}
%
%\begin{proof}
%We prove the claim by showing that a function $z$ as in (\ref{zeta}), satisfying the conditions of Definition \ref{countApprox}, also enjoys the properties of Assumption \ref{assumption}. First, by condition (i) in Definition \ref{countApprox} and by Eq. (\ref{rho}), one has:
%\[
%\Vert z_i \Vert \leq \rho M + \mu_{\theta} \leq M\text{,}
%\]
%hence condition (i) in Definition \ref{assumption} holds. Condition (ii) in Definition \ref{assumption} is simply implied by condition (ii) in Definition \ref{countApprox}, which concludes the proof.
%\end{proof}

We are now ready to present the main result of this section.

\begin{theorem}
\label{ThSpline}
Map $\mathcal{A}_{\mathcal{D}_{\tau}}$ in Definition \ref{countApprox} is a finite inner approximation of $\mathcal{D}_{\tau}$.
\end{theorem}

\begin{proof}
Consider any precision $\theta\in\mathbb{R}^{+}$. For notational simplicity we set $\rho_{\kappa_{d},\tau,M}(N_{\theta},\mu_{\hat{\theta}})=\rho$. As discussed in Remark \ref{remark2}, the set $\mathcal{A}_{\mathcal{D}_{\tau}}(\theta)$ is finite. Hence, condition (i) in Definition \ref{count_approx_def} is satisfied.  Condition (ii) in Definition \ref{count_approx_def} is implied by Lemma \ref{inner}. 
We now show that also condition (iii) in Definition \ref{count_approx_def} is satisfied. 
For any function $d\in\mathcal{D}_{\tau}$ consider a function $z$ as in (\ref{zeta}) where for any $i=0,1,...,N_{\theta}$ vectors $z_{i}$ are chosen in the set $2\mu_{\hat{\theta}}\mathbb{Z}^{l}$ such that:
%$z_{i}\in (2\mu_{\hat{\theta}}\mathbb{Z}^{l})\cap (\rho\,D)$ for any $i=0,1,...,N_{\theta}+1$ and
\begin{equation}
\Vert z_{i} - \rho\,d(ih) \Vert \leq \mu_{\hat{\theta}}.
\label{gio}
\end{equation}

We first prove that vectors $z_i$ are in the set $D$, showing that $\Vert z_i \Vert_{\infty} \leq M$ for all $i$. From Eq. (\ref{gio}), by using the triangle inequality and the definition of $\rho$ in (\ref{Rho}), one can write:

\begin{align*}
\Vert z_i \Vert_{\infty} & = \Vert z_i - \rho\,d(ih) + \rho\,d(ih)  \Vert_{\infty} \\
& \leq \Vert z_i - \rho\,d(ih)  \Vert_{\infty} + \Vert \rho\,d(ih)  \Vert_{\infty} \\
& \leq \mu_{\hat{\theta}} + \Vert \rho\,d(ih)  \Vert_{\infty} \\
& \leq \mu_{\hat{\theta}} + \rho\,M \\
& \leq \mu_{\hat{\theta}} + (1-\frac{\mu_{\hat{\theta}}}{M})\,M \\
& = \mu_{\hat{\theta}} + M - \mu_{\hat{\theta}} = M,
\end{align*}
which concludes the proof of the existence of such values $z_i\in (2\mu_{\hat{\theta}}\mathbb{Z}^{l})\cap D$, as in condition (i) of Definition \ref{countApprox}. %Next we show that function \mbox{$z\in \mathcal{A}_{\mathcal{D}_{\tau}}(\theta)$}. By definition of $z$, condition (i) in Definition \ref{countApprox} is satisfied. 
We now show that also condition (ii) is satisfied. From (\ref{gio}), the following chain of inequalities holds:
\[
\begin{array}
{rcl}
\Vert z_{i+1}-z_{i}\Vert & \leq & \Vert z_{i+1}-\rho d((i+1)h) \Vert + \Vert \rho d((i+1)h) - \rho d(ih)\Vert + \Vert \rho d(ih) -z_{i}\Vert \\
& \leq & \rho\Vert d((i+1)h)-d(ih)\Vert+2\mu_{\hat{\theta}}\\
& \leq & \rho\kappa_{d} h + 2\mu_{\hat{\theta}}\leq (1-\frac{2\mu_{\hat{\theta}}}{\kappa_{d} h})\kappa_{d} h + 2\mu_{\hat{\theta}}=\kappa_{d} h,
\end{array}
\]
where $h=\tau / (N_{\theta}+1)$ and the last inequality holds by the definition of function $\rho$ in (\ref{Rho}). Hence, condition (ii) in Definition \ref{countApprox} is satisfied and $z\in \mathcal{A}_{\mathcal{D}_{\tau}}(\theta)$. In order to conclude the proof of condition (iii) in Definition \ref{count_approx_def} we need to show that $\Vert d-z \Vert_{\infty}\leq \theta$. By the assumptions on the disturbance space, the following chain of inequalities holds:
\begin{align*}
\Vert d-z\Vert_{\infty} & = \max_{i=0,1,...,N_{\theta}, \, t\in [0,h]} \Vert d(ih+t)-z(ih+t) \Vert \\
& \leq \max_{i=0,1,...,N_{\theta}, \, t\in [0,h]}(\Vert d(ih+t)-\rho\,d(ih+t)\Vert 
+ \Vert \rho\,d(ih+t)-\rho\,d(ih)\Vert  \\
& \quad + \Vert\rho\,d(ih)-z(ih)\Vert + \Vert z(ih)-z(ih+t)\Vert ) \\
& \leq (1-\rho)M+(1+\rho)\kappa_{d} h+\mu_{\hat{\theta}}\\
& = \Theta_{\kappa_{d},\tau,M}(N_{\theta},\mu_{\hat{\theta}}) \leq \hat{\theta} \leq \theta,
\end{align*}
where the last step holds by Eq. (\ref{Theta}) and by definition of $N_{\theta}$ and $\mu_{\theta}$. From the above chain of inequalities, condition (iii) in Definition \ref{count_approx_def} is satisfied, which concludes the proof.
\end{proof}

\begin{remark}
Spline approximation of functional spaces has been also employed in \cite{PolaSCL10,PolaCDC10} for deriving symbolic models of nonlinear time--delay systems. The approximation scheme here proposed is different from the one proposed in \cite{PolaSCL10,PolaCDC10} as it can be readily seen by comparing Definition \ref{count_approx_def} and Definition 6 in \cite{PolaSCL10} (also employed in \cite{PolaCDC10}). In particular the notion of approximation here considered is stronger than the one used in \cite{PolaSCL10,PolaCDC10}, as it can be easily checked by comparing conditions (ii) in the two definitions. As discussed in the sequel, this notion allows us to provide symbolic models for nonlinear control systems affected by disturbances which can be \textit{effectively constructed} whereas the notion employed in \cite{PolaSCL10,PolaCDC10} does not. \\
\end{remark}

\section{Alternating approximately bisimilar symbolic models}\label{sec4}

In this section, we propose symbolic models that approximate nonlinear control systems with disturbances in the sense of alternating approximate bisimulation. \\
Given the control system $\Sigma=(X,\mathcal{U},\mathcal{D},f)$ in (\ref{tuple}) and a sampling time parameter $\tau\in{\mathbb{R}}^{+}$, consider the following transition system:
\begin{equation*}
T_{\tau}(\Sigma):=(X,\mathcal{U}_{\tau} \times \mathcal{D}_{\tau},\rTo_{\tau},O,H),
\label{systemTD}%
\end{equation*}
where:

\begin{itemize}
\item $\mathcal{U}_{\tau}=\{u\in\mathcal{U} | $ the domain of $u$ is $[0,\tau]$ and $u(t)=u(0)$, $t\in [0,\tau]$\};
\item $x\rTo_{\tau}^{(u,d)}x^{\prime}$ if there exists a trajectory $\xi:[0,\tau]\rightarrow X$ of $\Sigma$ satisfying \mbox{$\xi_{xud}(\tau)=x'$};
\item $O=X$;
\item $H=1_{X}$.
\end{itemize}

Transition system $T_{\tau}(\Sigma)$ is metric when we regard $O=X$ as being equipped with the metric \mbox{{\bfseries d}$(p,q)=\Vert p-q\Vert$}. 
Transition system $T_{\tau}(\Sigma)$ can be thought of as the time discretization of the control system $\Sigma$. For notational simplicity, in the following we denote by $u$ any constant control input $\tilde{u}$ s.t. $\tilde{u}(t)=u$ for all $t \in [0,\tau]$.
Consider a vector of quantization parameters 
\begin{equation}
\label{param_vec}
\mathbb{P}=(\tau,\mu_{x},\mu_{u},\mu_{d},N),
\end{equation}
and define the following transition system:

\begin{equation}
T_{\mathbb{P}}(\Sigma):=(Q_{\mathbb{P}},L_{\mathbb{P}},\rTo_{\mathbb{P}},O_{\mathbb{P}},H_{\mathbb{P}}), 
\label{symbmodel}
\end{equation}
where:

\begin{itemize}
\item $Q_{\mathbb{P}}=(2\mu_{x}\mathbb{Z}^{n}) \cap X$; 
\item $L_{\mathbb{P}}=A_{\mathbb{P}} \times B_{\mathbb{P}}$ where:
\begin{itemize}
\item[--] $A_{\mathbb{P}}=(2\mu_u\mathbb{Z}^{m}) \cap U$;
\item[--] $B_{\mathbb{P}}=\mathcal{A}_{\mathcal{D}_{\tau}}(\Theta_{\kappa_{d},\tau,M}(N, \mu_{d}))$ where $\mathcal{A}_{\mathcal{D}_{\tau}}$ is a finite inner approximation of $\mathcal{D}_{\tau}$, as in Definition \ref{countApprox}, and function $\Theta$ is defined as in (\ref{Theta});
\end{itemize}
\item $x\rTo_{\mathbb{P}}^{(u,d)}y$ if $\Vert \xi_{xud}(\tau) - y\Vert \leq \mu_{x}$;
\item $O_{\mathbb{P}}=X$;
\item $H_{\mathbb{P}}=\imath:Q_{\mathbb{P}}\hookrightarrow O_{\mathbb{P}}$.
\end{itemize}

\begin{remark}
It is readily seen that the transition system $T_{\mathbb{P}}(\Sigma)$ is countable and it becomes symbolic when the set of states $X$ is bounded. As stressed in Remark \ref{remark2}, the set of control and disturbance inputs $L_{\mathbb{P}}$ can be effectively computed from which the transition system $T_{\mathbb{P}}(\Sigma)$ can be effectively computed.  
\end{remark}

We now have all the ingredients to present the main result of this paper.

\begin{theorem}
\label{thmain}
Consider the control system $\Sigma=(X,\mathcal{U},\mathcal{D},f)$ in (\ref{tuple}) and suppose that:
\begin{itemize}
\item[(A1)] There exists a $\delta$--ISS Lyapunov function satisfying the inequality (ii) in Definition \ref{dISS_Lyapunov} for some $\lambda\in\mathbb{R}^{+}$.
\item[(A2)] There exists a $\mathcal{K}_{\infty}$ function $\gamma$ such that\footnote{Note that since $V$ is smooth, if the state space $X$ is bounded, which is the case as in many real applications, one can always choose $\gamma(\Vert w-z \Vert)= \left( \sup_{x,y\in X} \Vert \frac{\partial{V}}{\partial{y}}(x,y) \Vert \right) \Vert w-z \Vert$.}:
\[
V(x,x^{\prime})-V(x,x^{\prime \prime})\leq\gamma(\Vert{x^{\prime}-x^{\prime \prime}}\Vert),
\]
for every $x,x^{\prime},x^{\prime\prime}\in X$.
\end{itemize}
Then, for any desired precision $\varepsilon\in\mathbb{R}^{+}$, any sampling time $\tau\in\mathbb{R}^{+}$, and any choice of quantization parameters in $\mathbb{P}$ satisfying the following inequalities\footnote{Symbols $\hat{\mu}_X$, $\hat{\mu}_U$ and $\hat{\mu}_D$ are defined as in Section 2.1.}:
\begin{eqnarray}
&& \frac{\max \{ \sigma_{u} (\mu_{u}), \sigma_{d} (\theta_{d}) \}}{\lambda} + \frac{\gamma(\mu_{x})}{1-e^{-\lambda \tau}}\leq \underline{\alpha}(\varepsilon),\label{bisim_condition} \\
&& \mu_{x}\leq \hat{\mu}_X,\label{bisim_condition2} \\
&& \mu_{u}\leq \hat{\mu}_U,\label{bisim_condition3} \\
&& \mu_{d}\leq \hat{\mu}_D,\label{bisim_condition4} \\
&& \Theta_{\kappa_{d},\tau,M}(N,\mu_{d})\leq \theta_{d}, \label{bisim_condition5}
\end{eqnarray}
transition systems $T_{\tau}(\Sigma)$ and $T_{\mathbb{P}}(\Sigma)$ are alternating $\varepsilon$--approximately bisimilar. 
\label{polaut}
\end{theorem}

Before giving the proof of the above result we stress that:

\begin{proposition}
For any desired precision $\varepsilon\in\mathbb{R}^{+}$ and any sampling time $\tau\in\mathbb{R}^{+}$, there always exists a choice of the vector $\mathbb{P}$ of 
quantization parameters such that the coupled inequalities in (\ref{bisim_condition}), (\ref{bisim_condition2}), (\ref{bisim_condition3}), (\ref{bisim_condition4}) and (\ref{bisim_condition5}) are satisfied.
\end{proposition}

\begin{proof}
It is clear that a choice of sufficiently small parameters $\mu_x$, $\mu_u$ and $\theta_d$ allows to satisfy the inequalities in (\ref{bisim_condition})-(\ref{bisim_condition3}), since $\sigma_u$, $\sigma_d$ and $\gamma$ are $\mathcal{K}_{\infty}$ functions. Then, for any fixed $\theta_d$ resulting from the previous step, one can choose $N$ and $\mu_d$ such that the inequality in (\ref{bisim_condition5}) is fulfilled (as shown in the proof of Lemma \ref{inequality}), and finally $\mu_d$ can be chosen small enough so that the inequality in (\ref{bisim_condition4}) holds. 
%For the reader's convenience, an example of choice of the quantization parameters, obtained by inversion from the previous inequalities, is reported hereafter:
%\begin{align*}
%%\mu_x & =\min{} \\
%\mu_x & =\min\left\{\gamma^{-1} \left( \frac{\underline{\alpha}(\varepsilon)}{2} \left(1-e^{-\lambda \tau} \right) \right),\hat{\mu}_X\right\}, \\
%\mu_u & =\min\left\{\sigma_u^{-1} \left( \frac{ \lambda \underline{\alpha}(\varepsilon)}{2} \right),\hat{\mu}_U\right\}, \\
%\theta_d & =\sigma_d^{-1} \left( \frac{ \lambda \underline{\alpha}(\varepsilon)}{2} \right), \\
%N & = \left\lceil \max \left\{\frac{\kappa_{d} \tau }{2M},\frac{\theta_d \kappa_{d} \tau }{2M+2\kappa_{d}^2 \tau^2+\kappa_{d} \tau} \right\} \right\rceil, \\
%\mu_d & =\min\left\{\frac{1}{(N+1)^2},\hat{\mu}_D\right\}, \\
%\end{align*} 
%where the last two equations come from the definition of function $\Theta$ in (\ref{Theta}), and from the construction illustrated in the proof of Lemma  \ref{inequality}.
\end{proof}

We can now give the proof of Theorem \ref{thmain}.

\begin{proof}
Consider the relation $\mathcal{R}\subseteq X\times Q_{\mathbb{P}}$ defined by $(x,y)\in\mathcal{R}$ if and only if $V(x,y)\leq\underline{\alpha}(\varepsilon)$. 
Condition (i) in Definition \ref{ASR_S} is satisfied by the definition of $\mathcal{R}$ and condition (i) in Definition \ref{dISS_Lyapunov}. Let us now show that condition (ii) in Definition \ref{ASR_S} holds. 
Consider any $(x,y)\in\mathcal{R}$. By condition (\ref{bisim_condition3}), for any $u_{1}\in\mathcal{U}_{\tau}$ there exists $u_{2}\in A_{\mathbb{P}}=(2\mu_u\mathbb{Z}^{m}) \cap U$ such that:
\begin{equation}
\label{condU}
\Vert u_{2}-u_{1}\Vert_{\infty}\leq\mu_{u}.
\end{equation}
Moreover by Lemma \ref{inner} for any $d_{2}\in \mathcal{A}_{\mathcal{D}_{\tau}}(\theta_{d})$ we can pick $d_{1}=d_{2}\in \mathcal{D}_{\tau}$. Set $z=\xi_{yu_{2}d_{2}}(\tau)$. By condition (\ref{bisim_condition2}) there exists $v\in Q_{\mathbb{P}}$ such that: 
\begin{equation}
\Vert z- v \Vert \leq \mu_{x} .
\label{a2_1}
\end{equation}
Hence, by definition of $T_{\mathbb{P}}(\Sigma)$, the transition $y\rTo_{\mathbb{P}}^{u_{2},d_{2}}v$ is in $T_{\mathbb{P}}(\Sigma)$. Consider now the transition 
$x\rTo_{\tau}^{u_{1},d_{1}}w$ in $T_{\tau}(\Sigma)$. 
By Assumption (A1), condition (ii) in Definition \ref{dISS_Lyapunov} and the inequality in (\ref{condU}), one gets:
\[
\begin{array}
{rcl}
\frac{\partial{V}}{\partial{w}} f(w,u_{1},d_{2})+\frac{\partial{V}}{\partial{z}} f(z,u_{2},d_{2}) &
\leq & -\lambda V(w,z) + \sigma_{u}(\Vert{u_{1}-u_{2}}\Vert)+\sigma_{d}(\Vert{d_{1}-d_{2}}\Vert) \\
&
\leq & -\lambda V(w,z) + \sigma_{u}(\mu_{u}),
\end{array}
\]
which, by Assumption (A2), the definition of $\mathcal{R}$ and the inequality in (\ref{a2_1}), implies:
\begin{equation*}
\label{eq_chain_of_ineq}
\begin{split}
V(w,v) & \leq  V(w,z)+\gamma(\Vert{z-v}\Vert) \\
& \leq  V(w,z)+\gamma(\mu_{x}) \\
& \leq e^{-\lambda \tau} V(x,y)+(1-e^{-\lambda \tau})\frac{\sigma_{u} (\mu_{u})}{\lambda} + \gamma(\mu_{x}) \\
& \leq e^{-\lambda \tau} \underline{\alpha}(\varepsilon)+(1-e^{-\lambda \tau})\frac{\sigma_{u} (\mu_{u})}{\lambda} + \gamma(\mu_{x}).
\end{split}
\end{equation*}
Hence, by the inequality in (\ref{bisim_condition}), $V(w,v)\leq \underline{\alpha}(\varepsilon)$, from which $(w,v)\in\mathcal{R}$ and condition (ii) in Definition \ref{ASR_S} is proven. We now show condition (iii) in Definition \ref{ASR_S}. Consider any $(x,y)\in\mathcal{R}$. For any $u_{2}\in A_{\mathbb{P}}=(2\mu_u\mathbb{Z}^{m}) \cap U$ we can pick $u_{1}=u_{2}\in\mathcal{U}_{\tau}$. Consider any $d_{1}\in \mathcal{D}_{\tau}$. By Theorem \ref{ThSpline} and condition (\ref{bisim_condition4}) there exists 
$d_{2}\in \mathcal{A}_{\mathcal{D}_{\tau}}(\theta_{d})$ such that: 
\begin{equation}
\label{condD}
\Vert d_{2}-d_{1}\Vert_{\infty}\leq \Theta_{\kappa_{d},\tau,M}(N,\mu_{d})\leq \theta_{d}.
\end{equation}
Set $z=\xi_{yu_{2}d_{2}}(\tau)$. By condition (\ref{bisim_condition2}) there exists $v\in Q_{\mathbb{P}}$ such that the inequality in (\ref{a2_1}) holds true.
Hence, by definition of $T_{\mathbb{P}}(\Sigma)$, the transition $y\rTo_{\mathbb{P}}^{u_{2},d_{2}}v$ is in $T_{\mathbb{P}}(\Sigma)$. Consider now the transition 
$x\rTo_{\tau}^{u_{1},d_{1}}w$ in $T_{\tau}(\Sigma)$. 
By Assumption (A1), condition (ii) in Definition \ref{dISS_Lyapunov} and the inequality in (\ref{condD}), one gets:
\[
\begin{array}
{rcl}
\frac{\partial{V}}{\partial{w}} f(w,u_{1},d_{2})+\frac{\partial{V}}{\partial{z}} f(z,u_{2},d_{2}) &
\leq & -\lambda V(w,z) + \sigma_{u}(\Vert{u_{1}-u_{2}}\Vert)+\sigma_{d}(\Vert{d_{1}-d_{2}}\Vert) \\
&
\leq & -\lambda V(w,z) + \sigma_{d}(\theta_{d}),
\end{array}
\]
which, by Assumption (A2), the definition of $\mathcal{R}$ and the inequality in (\ref{a2_1}), implies:
\begin{equation*}
\begin{split}
V(w,v) & \leq  V(w,z)+\gamma(\Vert{z-v}\Vert) \\
& \leq  V(w,z)+\gamma(\mu_{x}) \\
& \leq e^{-\lambda \tau} V(x,y)+(1-e^{-\lambda \tau})\frac{\sigma_{d} (\theta_{d})}{\lambda} + \gamma(\mu_{x}) \\
& \leq e^{-\lambda \tau} \underline{\alpha}(\varepsilon)+(1-e^{-\lambda \tau})\frac{\sigma_{d} (\theta_{d})}{\lambda} + \gamma(\mu_{x}).
\end{split}
\end{equation*}
Hence, by the inequality in (\ref{bisim_condition}), $V(w,v)\leq \underline{\alpha}(\varepsilon)$, from which $(w,v)\in\mathcal{R}$ and condition (iii) in Definition \ref{ASR_S} is proven. 
Finally by definition of $\mathcal{R}$ it is easy to see that $\mathcal{R}(X)=Q_{\mathbb{P}}$ and $\mathcal{R}^{-1}(Q_{\mathbb{P}})=X$.
\end{proof}

\section{Control design of a pendulum}\label{sec5}

In this section, we consider a slight variation of the classical pendulum  model \cite{khalil} where the point mass is subject to a horizontal acceleration, modeling e.g. the wind. The resulting dynamics is described by:
\begin{equation*}
\Sigma:
\left\{
\begin{array}{clrr}
& \dot{x}_{1}=x_{2},\\
& \dot{x}_{2}=- \frac{g}{l} \sin{x_{1}}- \frac{k}{m} x_2 + \frac{1}{m l^2} u+ d \cos{x_{1}},
\end{array}
\right.
\label{example}
\end{equation*}
where $x_{1}$ and $x_{2}$ are the angular position and velocity of the point mass, $u$ is the torque representing the control variable, $d$ is the (unknown) horizontal acceleration, \mbox{$g=9.8$} is the gravity acceleration, $l=0.5$ is the length of the rod, $m=0.6$ is the mass of the bob, $k=2$ is the coefficient of friction. All constants and variables in $\Sigma$ are expressed in the International System. 
We assume $X=X_1 \times X_2$, $U=[\underline{u},\overline{u}]$ and $D=[\underline{d},\overline{d}]$, with $X_1=[-\pi/4,\pi/4]$, $X_2=[-0.5,0.5]$, $\overline{u}=-\underline{u}=1.5$, $\underline{d}=-0.01$ and $\overline{d}=0.02$. We first construct a symbolic model for $\Sigma$. To this aim we apply Theorem \ref{thmain}. As a first step, we need to show that the control system $\Sigma$ is $\delta$--ISS. Consider the following candidate quadratic $\delta$--ISS Lyapunov function:
\[
V(x,y)=(x-y)^{^{\prime}}
\begin{bmatrix}  1.5 & 0.3 \\ 0.3 & 1.5
\end{bmatrix}
(x-y) .
\]
It is possible to show that $V$ satisfies condition (i) of Definition \ref{dISS_Lyapunov} with 
\[
\begin{array}
{rcl}
\underline{\alpha}(r)=1.2\,r^2, & \overline{\alpha}(r)=3.6\, r^2, & r\in\mathbb{R}^{+}_{0}. 
\end{array}
\]
Moreover, it is possible to show that:
\[
\frac{\partial V}{\partial x_{1}}f(x_{1},u_{1},d_{1})
+\frac{\partial V}{\partial x_{2}}f(x_{2},u_{2},d_{2})  \leq -0.77\,V(x_{1},x_{2})+8.76\,\Vert u_{1}-u_{2}\Vert +1.31\,\Vert d_{1}-d_{2}\Vert,
\]
from which condition (ii) of Definition \ref{dISS_Lyapunov} is fulfilled with $\lambda=0.77$, $\sigma_u(r)=8.76\,r$ and $\sigma_d(r)=1.31\,r$, $r\in\mathbb{R}^{+}_{0}$. We consider disturbance inputs with Lipschitz constant $\kappa_{d}=0.002$. 
For a chosen precision $\varepsilon=0.125$, the inequality in (\ref{bisim_condition}) is satisfied with parameters 
\[
\begin{array}
{llll}
\tau=1, & \mu_{x}=\pi/2000, & \mu_{u}=0.001, & \theta_{d}=0.007. 
\end{array}
\]

Lemma \ref{inequality} ensures existence of parameters $\mu_{d}$ and $N$ satisfying the inequality: 
\[
\Theta_{\kappa_{d},\tau,\overline{d}}(N,\mu_d)\leq \theta_d .
\]
One possible choice of such parameters is \mbox{$\mu_d=1.43 \cdot 10^{-4}$} and $N=0$; the choice of the last parameter implies that the functional space $\mathcal{D}_{\tau}$ is approximated by two splines. 
The resulting symbolic model $T_{\mathbb{P}}(\Sigma)$ in (\ref{symbmodel}) has been constructed and consists of $159,819$ states, $1,501$ control inputs and $6,366$ disturbance inputs. The running time needed for computing $T_{\mathbb{P}}(\Sigma)$ is $4,679$s using a laptop with CPU Intel Core 2 Duo T5500 @ 1.66 GHz with 4 GB RAM. We do not report in the paper further details on $T_{\mathbb{P}}(\Sigma)$ because of its large size. 
Instead, we use the obtained symbolic model to solve the following robust control design problem with synchronization specifications on the angular position of the pendulum:
\begin{itemize}
\item starting from $x_{0}=(0,0)$, reach $\Omega_{1}=[ \pi/8, \pi/4]\times X_2$;
\item stay in $\Omega_{1}$ for a time duration between $2$s and $4$s;
\item reach $\Omega_{2}=[- \pi/4, - \pi/8]\times X_2$;
\item stay in $\Omega_{2}$ for at most $3$s;
\item go back to $\Omega_{1}$ and stay definitively in $\Omega_{1}$.
\end{itemize}
Such a specification is a simple example of more complex specifications that typically arise in multi--agent systems where (space) resources are shared in order to perform a cooperative task. 
By using standard fixed--point algorithms (see e.g. \cite{paulo}) we designed the symbolic controller enforcing the prescribed specification. The resulting controller has been constructed in $2,681s$ with a memory occupation of $716$ integers. For the disturbance input realization
\[
d(t)=\frac{\overline{d}-\underline{d}}{2}\cos\left( \frac{2\kappa_{d}}{\overline{d}-\underline{d}}t \right)+\frac{\overline{d}+\underline{d}}{2},
\]
the specification is shown in Figure \ref{state} to be satisfied, by means of the symbolic control law illustrated in Figure \ref{input}.
\begin{figure}
\begin{center}
\includegraphics[scale=0.45]{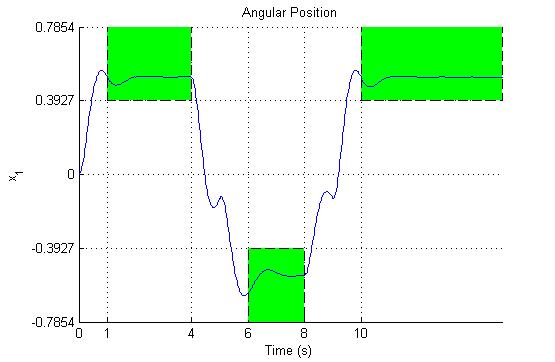}
\caption{Angular trajectory of the pendulum.} 
\label{state}
\end{center}
\end{figure}
\begin{figure}
\begin{center}
\includegraphics[scale=0.45]{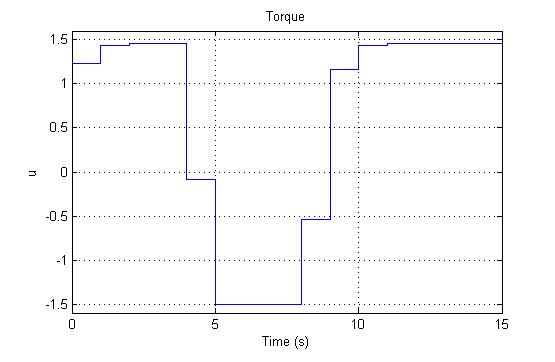}
\caption{Symbolic control input.} 
\label{input}
\end{center}
\end{figure}

\section{Conclusion}
\label{sec6}
In this paper we showed how to construct symbolic models that approximate nonlinear control systems affected by disturbances. 
%The results presented in this paper provide an important improvement upon the results reported in \cite{PolaSIAM2009} because they propose symbolic models that can be effectively computed. 
Future work will focus on algorithms for the construction of the symbolic models presented in this paper. 
%The computational complexity erasing in the computation of the proposed symbolic models is generally high. Future work will focus on efficient algorithms for the implementation of the symbolic models presented in this paper. 
%Despite previous results in this research topic \cite{PolaSIAM2009}, which require the computation of the set of reachable states, the proposed symbolic models are based on a first-order spline approximation of the control and disturbance input spaces and are shown to be effectively computable. 
%Further work will focus on efficient algorithms which cope with the computational complexity of the proposed approach. Useful insights in this direction can be found in the results reported in \cite{BorriCDC2010}. 

\bigskip

\textbf{Acknowledgment.}
The authors are grateful to Pierdomenico Pepe, Paulo Tabuada and Antoine Girard for fruitful discussions on the topic of the present paper.

\bibliographystyle{alpha}
\bibliography{biblio1,reference}

\end{document}